\documentclass[a4paper, 11pt]{amsart}
\usepackage{a4wide}
\usepackage[english]{babel} 
\usepackage{amsmath, amssymb} 
\usepackage{tikz}
\usepackage{tikz-cd}
	\tikzcdset{row sep/normal=1.3em}
	\tikzcdset{column sep/normal=1.3em}
\usetikzlibrary{decorations}
\usetikzlibrary{shapes,snakes}
\usepackage[normalem]{ulem}
\usepackage[mathscr]{euscript}

\usepackage{hyperref}
\usepackage[capitalize,noabbrev]{cleveref}

\DeclareFontFamily{T1}{cbgreek}{}
\DeclareFontShape{T1}{cbgreek}{m}{n}{<-6>  grmn0500 <6-7> grmn0600 <7-8> grmn0700 <8-9> grmn0800 <9-10> grmn0900 <10-12> grmn1000 <12-17> grmn1200 <17-> grmn1728}{}
\DeclareSymbolFont{quadratics}{T1}{cbgreek}{m}{n}
\DeclareMathSymbol{\qoppa}{\mathord}{quadratics}{19}
\DeclareMathSymbol{\Qoppa}{\mathord}{quadratics}{21}

\newcommand{\id}{\mathrm{id}}
\newcommand{\Perf}{\mathrm{Perf}}
\newcommand{\QF}{\Qoppa}
\newcommand{\q}{\mathrm{q}}
\newcommand{\s}{\mathrm{s}}
\newcommand{\gs}{\mathrm{gs}}
\renewcommand{\t}{\mathrm{t}}
\renewcommand{\b}{\mathrm{b}}
\newcommand{\lto}{\longrightarrow}
\renewcommand{\L}{\mathrm{L}}
\newcommand{\KL}{\mathbb{L}}
\newcommand{\GW}{\mathrm{GW}}
\newcommand{\KGW}{\mathbb{GW}}
\newcommand{\ko}{\mathrm{ko}}
\newcommand{\ku}{\mathrm{ku}}
\newcommand{\KO}{\mathrm{KO}}
\newcommand{\MSO}{\mathrm{MSO}}
\newcommand{\cwedge}{{\scriptscriptstyle\wedge}}
\newcommand{\Z}{\mathbb{Z}}

\newcommand{\F}{\mathbb{F}}
\newcommand{\C}{\mathscr{C}}
\newcommand{\E}{\mathbb{E}}
\renewcommand{\S}{\mathbb{S}}
\newcommand{\map}{\mathrm{map}}
\newcommand{\op}{\mathrm{op}}
\newcommand{\Sp}{\mathrm{Sp}}
\newcommand{\Spc}{\mathrm{Spc}}
\newcommand{\Cat}{\mathrm{Cat}}

\DeclareMathOperator*{\colim}{colim}

\newtheorem{Thm}{Theorem}
\newtheorem{Cor}[Thm]{Corollary}

\newtheorem{Prop}[Thm]{Proposition}

\theoremstyle{definition}

\newtheorem{Example}[Thm]{Example}
\newtheorem*{Setup}{Setup}
\newtheorem*{Acknowledgements}{Acknowledgements}

\theoremstyle{remark}

\newtheorem*{claim*}{Claim}

\newtheorem{Rmk}[Thm]{Remark}

\theoremstyle{plain}
\newcounter{zaehler}

\newtheorem*{introcor*}{Corollary}

\title{Remarks on chromatically localised hermitian K-theory}

\author{Markus Land}
\address{Mathematisches Institut, Ludwig-Maximilians-Universit\"at M\"unchen, Theresienstra\ss e 39, 80333 M\"unchen, Germany}
\email{markus.land@math.lmu.de}

\thanks{The author was supported by the Danish National Research Foundation through the Copenhagen Centre for Geometry and Topology (DNRF151).}

\date{\today}

\begin{document}

\begin{abstract}
We describe chromatic localisations of genuine L-spectra of discrete rings and deduce that the purity property of $K(1)$-local $K$-theory of rings established by Bhatt--Clausen--Mathew also holds in Grothendieck--Witt theory. In addition, we collect some results on higher chromatic localisations of various L-theory spectra and their consequences for Grothendieck--Witt theory.
\end{abstract}

\maketitle

The recent times have shown many advances in chromatically localised algebraic $K$-theory, notably its relation to \'etale $K$-theory, redshift, and purity phenomena \cite{AKS,AKQ,BCM,CM,CMNN,HW,LMMT,Mathew,Yuan}. In addition, in the series of papers \cite{CDHI, CDHII, CDHIII, CDHIV} we have shown that Grothendieck--Witt theory, also known as hermitian K-theory, sits in a fibre sequence with algebraic $K$-theory and $\L$-theory. The purpose of this short note is to collect some results on chromatically localised L-theory and to use this to lift some of the purity and redshift results from algebraic $K$-theory to Grothendieck--Witt theory. 

\section{Discrete rings}

\begin{Setup}
Let $R$ be a discrete ring equipped with an invertible module with involution $M$ over $R$ in the sense of \cite[Definition 3.1.1 \& 3.1.4]{CDHI}. For instance, $R$ could be a commutative ring (or more generally a ring with an anti-involution) in which case we may choose $M=R$ with involution given by $\pm \id$. Recall for instance from \cite[Definition R.3]{CDHIII} that in this situation, we have the genuine Poincar\'e structures $\QF^{\geq m}_M$ on $\Perf(R)$, which come with comparison maps 
\[ \QF^\q_M \lto \dots \QF^{\geq m+1}_M \lto  \QF^{\geq m}_M \lto \QF^{\geq m-1}_M \lto \dots \lto \QF^\s_M.\]
where $\QF^\q_M = \QF^{\geq \infty}_M$ and $\QF^\s_M = \QF^{\geq -\infty}_M$. For any Poincar\'e structure $\QF$ on $\Perf(R)$, we will write $\L(R;\QF)$ for the L-spectrum of the Poincar\'e category $(\Perf(R),\QF)$, and we will further write $\L^\q(R;M)$ for $\L(R;\QF^\q_M)$ and $\L^\s(R;M)$ for $\L(R;\QF^\s_M)$. Also, we will use the notation $\L^\gs(R;M)$ for $\L(R;\QF^{\geq 0}_M)$ and simply $\L^\gs(R)$ if $R$ is commutative and $M= R$ with trivial involution; This version is called \emph{genuine symmetric} L-theory, since the Poincar\'e structure $\QF^{\geq 0}_M$ is the non-abelian derived structure associated to classical $M$-valued symmetric bilinear forms, see \cite[Section 4.2]{CDHI}. Using algebraic surgery, we have shown in \cite{CDHIII} that genuine symmetric L-theory recovers Ranicki's very first definition of (non-periodic) symmetric L-theory which appears in his work on the algebraic theory of surgery \cite{Ranicki-FoundationsI, Ranicki-FoundationsII}. Moreover, inspired by results of Galatius--Randal-Williams \cite{GRW1, GRW2}, Hebestreit--Steimle \cite{HS} use parametrised algebraic surgery to show that the connective cover of $\GW(R;\QF^{\geq 0}_M)$ is equivalent to the group completion of the symmetric monoidal category of unimodular symmetric $M$-valued bilinear forms, thereby relating the general Grothendieck--Witt spectra of \cite{CDHII} to classical hermitian $K$-theory. This result has been exploited in \cite{CDHIII}.
\newline

In this note, we work with the telescopic localisation functors $L_{T(n)}$, see for instance \cite[Section 2]{LMMT} for a summary of the basic properties we use here. However, all results in this note hold equally well for the chromatic localisations $L_{K(n)}$ in place of $L_{T(n)}$, where $K(n)$ is the Morava $K$-theory spectrum of height $n$. At height one, one has $L_{T(1)} = L_{K(1)}$ and we shall write $L_{K(1)}$ rather than $L_{T(1)}$, simply for sake of familiarity. For higher heights, the localisation map $\id \to L_{K(n)}$ factors through a canonical map $L_{T(n)} \to L_{K(n)}$, making $L_{T(n)}$ an a priori stronger localisation functor. The telescope conjecture, however, asserts that the map $L_{T(n)} \to L_{K(n)}$ is an equivalence, but is widely open in heights $\geq 2$. We refer to \cite{Barthel} for a survey and to \cite{CSY, CSY2} for general and categorical properties of the $T(n)$- and $K(n)$-local categories of spectra. Finally, we recall that all telescopic and chromatic localisation functors depend on an implicit prime $p$.
\end{Setup}

\begin{Thm}\label{Theorem1}
We have
\[ L_{K(1)}\L(R,\QF^{\geq m}_M) \simeq 
	\begin{cases} 0 & \text{ if } p = 2, \\
		\L^\s(R;M)^\cwedge_p & \text{ if } p \neq 2 
	\end{cases} \]
In addition, $L_{T(n)}\L(R;\QF^{\geq m}_M)=0$ for all $n\geq 2$ and all primes $p$.
\end{Thm}

For odd primes, we interpret this theorem as saying that the genuine L-theory of a discrete ring is not far away from being $K(1)$-local. In fact, \cref{Theorem1} and \cite[Corollary 1.3.9]{CDHIII} imply that for coherent rings of finite global dimension $d$, the map
\[ \L(R;\QF^{\geq m}_M) \lto L_{K(1)} \L(R;\QF^{\geq m}_M) \]
is $p$-adic equivalence in degrees greater or equal to $d+m-1$. We view this as an L-theoretic Quillen--Lichtenbaum type theorem (which is, of course, much easier than its $K$-theoretic counterpart). In contrast to the case of $K$-theory, we note here that $p$-adically completed L-theory, and therefore by \cref{Theorem1} also $K(1)$-local L-theory for odd primes, does \emph{not} satisfy Galois descent\footnote{Consider for instance the Galois extensions $\mathbb{R} \to \mathbb{C}$ at odd primes or $\F_2 \to \F_4$ at prime 2.}. In particular, unlike in $K$-theory, there is no good comparison between $K(1)$-localised L-theory and the ($p$-adic) \'etale sheafification of L-theory.

\begin{proof}[Proof of \cref{Theorem1}]
We first show the vanishing statements of the theorem. For this we observe that $\L(R;\QF^{\geq m}_M)$ is a module spectrum over $\L^\gs(\Z)$, so it suffices to show that $L_{T(n)}\L^\gs(\Z)$ vanishes for $n \geq 2$ at all primes and for $n=1$ at prime 2.

We first consider the case $p=2$. It will suffice to show that $\L^\gs(\Z)_{(2)}$ is an $H\Z$-module, since $H\Z$ is $T(n)$-acyclic for $n\geq 1$. To see this, we show that $\L^\gs(\Z)_{(2)}$ receives an $\E_2$-ring map\footnote{In fact, this map refines to an $\E_3$-ring map, see \cite[Remark 3.10]{HLN}.} from $H\Z$. Indeed, we recall from \cite{LMcC} that there is an $\E_\infty$-map $\MSO \to \L^\s(\Z)$. Since $\MSO$ is connective and $\L^\gs(\Z) \to \L^\s(\Z)$ is an equivalence on connective covers by \cite[Corollary 1.3.10]{CDHIII}, one also obtains an $\E_\infty$-map $\MSO \to \L^\gs(\Z)$. Now it follows from the Hopkins--Mahowald theorem describing $H\Z_{(2)}$ as a $\E_2$-Thom-spectrum that there is an $\E_2$-map $H\Z \to \MSO_{(2)}$. In total, we obtain an $\E_2$-map $H\Z \to \L^\gs(\Z)$ as needed, see also \cite[Section 3]{HLN} for more details. We note that the fact that $\L^\s(\Z)_{(2)}$ is a module spectrum over $H\Z$ is due to Taylor--Williams \cite{TW}. Away from prime $2$, the situation is different, but still well understood: There is a canonical equivalence of $\E_\infty$-rings $\KO[\tfrac{1}{2}] \simeq \L^{\gs}(\Z)[\tfrac{1}{2}]$ by the results of \cite{LN} together with the fact that the map of $\E_\infty$-ring spectra $\L^\gs(\Z)[\tfrac{1}{2}] \to \L^\s(\Z)[\tfrac{1}{2}]$ is an equivalence \cite[Proposition 3.1.14]{CDHIII}; the fact that there is an equivalence of homotopy ring spectra $\KO[\tfrac{1}{2}] \simeq \L^\s(\Z)[\tfrac{1}{2}]$ was known long before and is due to Sullivan. Since $\KO$ is $K(n)$-acyclic for $n\geq 2$ (it is essentially $K(1)$-local), the vanishing statements follow.

To see the remaining claim, we again use \cite[Proposition 3.1.14]{CDHIII}, which shows that the map
\[ \L(R;\QF^{\geq m}_M)[\tfrac{1}{2}] \lto \L^\s(R;M)[\tfrac{1}{2}] \]
is an equivalence. Hence, the map
\[ \L(R;\QF^{\geq m}_M) \lto \L^\s(R;M)\]
is a $p$-adic equivalence for odd primes $p$. It therefore suffices to argue why $\L^\s(R;M)^\cwedge_p$ is $K(1)$-local.
For this, we recall that there is an equivalence of $\E_\infty$-rings $\KO^\cwedge_p \simeq \L^\s(\Z)^\cwedge_p$. This shows that $\L^\s(R;M)^\cwedge_p$ is a $p$-complete $\KO^\cwedge_p$-module. For the convenience of the reader, we record here the well-known fact that $p$-complete $\KO$-modules $X$ are $K(1)$-local. So let $Y$ be a $K(1)$-acyclic spectrum and recall that $K(1)$ and $\KO/p$ have the same Bousfield class so that $Y$ is also $\KO/p$-acyclic. We then obtain
\begin{align*}
\map(Y,X) & \simeq \map_{\KO}(\KO\otimes Y,X) \\
	& \simeq \lim_n \map_{\KO}(\KO\otimes Y,X/p^n) \\
	& \simeq \lim_n \map_{\KO}(\Omega (\KO\otimes Y)/p^n,X) \\
	& \simeq \lim_n \map_{\KO}(\Omega (\KO/p^n \otimes Y),X).
\end{align*}
Now in the final term, for fixed $n$, the mapping space vanishes, as one sees inductively by the assumption that $\KO/p \otimes Y = 0$. Hence $X$ is $K(1)$-local as needed.
\end{proof}

\begin{Rmk}
We note here a consequence of the just proven fact that $\L^\s(R;M)^\cwedge_p$ is $K(1)$-local. Recall that the Bousfield--Kuhn functor at height $n$ is a functor $\Phi_n \colon \Spc_* \to \Sp$ with the property that $L_{T(n)} = \Phi_n \circ \Omega^\infty$. For $n=1$ this functor was found by Bousfield, and it was later constructed for all heights by Kuhn \cite{Kuhn}. Therefore, we deduce that for odd primes, the $p$-completion of $\L^\s(R;M)$ can be recovered from $\Omega^\infty \L^\s(R;M)$ (as a space, not as an $\E_\infty$-space). In addition, the $2$-localisation can also be recovered from $\Omega^\infty \L^\s(R;M)$ since $\L^\s(R;M)_{(2)}$ is an Eilenberg--Mac Lane spectrum and periodic. It follows that the spectrum $\L^\s(R;M)$ can be (integrally) recovered from the space $\Omega^\infty \L^\s(R;M)$.
\end{Rmk}

\begin{Prop}\label{Cor1}
Let $R$ and $M$ be as above, and let $p$ be a prime. The map
\[ \GW(R;\QF^{\geq m}_M) \lto \GW^\s(R;M) \]
is a $K(1)$-local equivalence at all primes and a $p$-adic equivalence at odd primes. Likewise, $\GW(R;\QF^{\geq m}_M)$ vanishes $T(n)$-locally for $n\geq 2$.
\end{Prop}
\begin{proof}
As a consequence of \cite[Corollary 4.4.14]{CDHII}, there is a pullback square
\[ \begin{tikzcd}
	\GW(R;\QF^{\geq m}_M) \ar[r] \ar[d] & \L(R;\QF^{\geq m}_M) \ar[d] \\
	\GW^\s(R;M) \ar[r] & \L^\s(R;M)
\end{tikzcd}\]
and the right vertical map is a $p$-adic equivalence for odd primes, as was shown in the proof of \cref{Theorem1}, and is a
$K(1)$-local equivalence by \cref{Theorem1} for all primes. To see the second part, we again consider the fundamental fibre sequence \cite[Corollary 4.4.14]{CDHII}
\[ K(R)_{hC_2} \lto \GW(R;\QF^{\geq m}_M) \lto \L(R;\QF^{\geq m}_M)\]
and use \cref{Theorem1} and that $L_{T(n)}(K(R)_{hC_2}) \simeq L_{T(n)}\big((L_{T(n)}K(R))_{hC_2}\big)$ which vanishes for $n\geq 2$ by Mitchell's theorem \cite{Mitchell}.
\end{proof}

In addition, we have the following invariance result for $p$-adically completed L-theory, see also \cite[Proposition 3.6.4]{Ranicki-yellow-book}.
\begin{Prop}\label{Prop4}
Let $R$ and $M$ be as above, and let $p$ be a prime. Then the map 
\[ \L(R;\QF^{\geq m}_M) \lto \L(R[\tfrac{1}{p}];\QF^{\geq m}_M) \]
is an equivalence after inverting 2. In particular it is a $p$-adic equivalence for odd primes $p$.
\end{Prop}
\begin{proof}
If the $p$-primary torsion of $R$ is bounded, \cite[Proposition 2.3.6]{CDHIII} gives a localisation-completion pullback square
\[ \begin{tikzcd}
	\L(R;\QF^{\geq m}_M) \ar[r] \ar[d] & \L(R^\cwedge_p;\QF^{\geq m}_M) \ar[d] \\
	\L(R[\frac{1}{p}];\QF^{\geq m}_M) \ar[r] & \L(R^\cwedge_p[\tfrac{1}{p}];\QF^{\geq m}_M)
\end{tikzcd}\]
in which the terms on the right hand side are module spectra over $\L^\gs(\Z^\cwedge_p) = \L(\Z^\cwedge_p;\QF^{\geq 0}_\Z)$. The assumption on the boundedness of the torsion was only made in loc.\ cit.\ to ensure that the symbol $R^\cwedge_p$ is unambiguous; under the boundedness assumption on the $p$-primary torsion, the algebraic completion and the derived completion agree. The above localisation-completion square holds true unconditionally for the derived completion, and the module structure arguments carry over to this case as well.
It hence suffices to show that $\L(\Z^\cwedge_p;\QF^{\geq 0}_\Z)$ vanishes after inverting 2, which follows from the fact that it is a ring spectrum whose $\pi_0$ is 2-primary torsion:
\[ \L_0^\gs(\Z^\cwedge_p) \cong \begin{cases} \Z/8 \oplus \Z/2 & \text{ if } p = 2, \\
									\Z/2\oplus \Z/2 & \text{ if } p \equiv 1 \mod 4,  \\
									\Z/4 & \text{ if } p \equiv 3 \mod 4 \end{cases} \]
see \cite[pg.\ 374]{Ranicki-yellow-book}, using \cite[Corollary 1.2.16]{CDHIII}.
\end{proof}

We deduce the following corollary, which is the analog of the invariance result of Bhatt--Clausen--Mathew \cite{BCM} for $K(1)$-local $K$-theory which was then coined \emph{purity} in \cite{LMMT}.
\begin{Cor}\label{Cor5}
Let $R$ and $M$ be as above and let $p$ be prime. Then the map 
\[ \GW(R;\QF^{\geq m}_M) \lto \GW(R[\tfrac{1}{p}];\QF^{\geq m}_M) \]
is a $K(1)$-local equivalence.
\end{Cor}
\begin{proof}
We consider the diagram of fibre sequences
\[\begin{tikzcd}
	K(R)_{hC_2} \ar[r] \ar[d] & \GW(R;\QF^{\geq m}_M) \ar[r] \ar[d] & \L(R;\QF^{\geq m}_M) \ar[d] \\
	K(R[\tfrac{1}{p}])_{hC_2} \ar[r] & \GW(R[\tfrac{1}{p}];\QF^{\geq m}_M) \ar[r] & \L(R[\tfrac{1}{p}];\QF^{\geq m}_M)
\end{tikzcd}\]
in which the left vertical map is an equivalence by \cite{BCM, Mathew, LMMT}; notice that in the references this result is sometimes phrased with non-connective $K$-theory instead, but the $K(1)$-localisation does not see the difference between the two versions. By \cref{Prop4}, the right vertical map is a $p$-adic equivalence for odd $p$, and hence also a $K(1)$-local equivalence. If $p=2$ then both terms appearing on the right are $K(1)$-acyclic by \cref{Theorem1}, and hence the map is also a $K(1)$-local equivalence.
\end{proof}

\begin{Rmk}
For $p=2$, the map 
\[ K(R)_{hC_2} \lto \GW(R;\QF^{\geq m}_M) \]
is a $K(1)$-local equivalence, since the L-theory term vanishes $K(1)$-locally. Similarly, it follows that the map $\GW^\s(R;M) \to K(R)^{hC_2}$ is a $K(1)$-local equivalence since $K(R)^{tC_2}$ also vanishes $K(1)$-locally, as it is also a module over $\L^\gs(\Z)$. In other words, $K(1)$-locally, the homotopy limit problem is true for any ring.
\end{Rmk}
\begin{Rmk}
For any prime $p$, the square
\[\begin{tikzcd}
	K(R)_{hC_2} \ar[r] \ar[d] & \GW(R;\QF^{\geq m}_M) \ar[d] & \\
	K(R[\tfrac{1}{p}])_{hC_2} \ar[r] & \GW(R[\tfrac{1}{p}];\QF^{\geq m}_M) 
\end{tikzcd}\]
is a pullback after inverting 2 by \cref{Prop4}.
\end{Rmk}

\cref{Theorem1} also has consequences for other L-spectra, not associated to the genuine family of Poincar\'e structure $\QF^{\geq m}_M$ considered so far. For instance, we can also treat the spectrum $\L^\b(\Z)$ obtained from the Burnside Poincar\'e structure $\QF^\b$ on $\Perf(\Z)$, appearing in \cite[Corollary 1.2.24]{CDHIII} and \cite{DO}. This Burnside Poincar\'e structure is a non-abelian derived structure, and the results we state next apply more generally to any non-abelian derived structure on $\Perf(\Z)$, for instance those associated to a form parameter in the sense of Bak, and more generally in the sense of Schlichting \cite{Schlichting}, see \cite[Definition 4.2.26]{CDHI} for a definition in the language of Poincar\'e categories. Recall the definition of an $r$-symmetric Poincar\'e structure from \cite[Definition 1.1.2]{CDHIII} and the fact that non-abelian derived structures are always 0-symmetric \cite[Remark 1.3.17]{CDHIII}.

\begin{Cor}\label{Cor:finite-dimensional-rings}
Let $R$ be a coherent ring of finite global dimension and let $\QF$ be an $r$-symmetric Poincar\'e structure on $\Perf(R)$ with underlying module with involution $M$. Then we have
\[ L_{K(1)}\L(R;\QF) \simeq \begin{cases} 0 & \text{ if } p=2, \\ \L^\s(R;M)^\cwedge_p  & \text{ if } p \neq 2 \end{cases}\]
In addition, $L_{T(n)}\L(R;\QF)$ vanishes for $n\geq 2$.
\end{Cor}
\begin{proof}
In \cite[Corollary 1.3.9]{CDHIII}, it is shown that the map 
\[ \L(R;\QF) \lto \L^\s(R;M) \]
has bounded above fibre and is therefore a $T(n)$-local equivalence for all $n\geq 1$. The result thus follows from \cref{Theorem1}.
\end{proof}

\begin{Rmk}
In particular, we deduce that $L_{T(n)}\L^\b(\Z)$ vanishes for $n \geq 2$ and all primes, and for $n=1$ at prime 2. In contrast to $\L^\gs(\Z)$, we do not expect $\L^\b(\Z)$ to be 2-adically an $H\Z$-algebra or $p$-adically (for odd primes) a $\ko$-algebra.
\end{Rmk}

\cref{Theorem1} also has consequences for not necessarily non-abelian derived Poincar\'e structures on $\Perf(R)$: A classical example one has here is the Tate Poincar\'e structure $\QF^\t$ for a commutative ring $R$, see \cite[Example 3.2.12]{CDHI}. We will prove a more general result later, and hence defer the proof of the following proposition to later. 

\begin{Prop}\label{Prop-TF}
Let $R$ be a commutative ring. Then the canonical map $\L^\q(R) \to \L^\t(R)$ is a $T(n)$-local equivalence for $n\geq 1$ and all primes.
\end{Prop}
\begin{proof}
This is a special case of \cref{Prop-general}.
\end{proof}

\begin{Rmk}
Again, \cref{Prop-TF} implies that $\L^\t(R)$ is $K(1)$-acyclic at prime 2 and $T(n)$ acyclic for $n\geq 2$ and all primes, but we expect that $\L^\t(R)$ is not 2-adically an $H\Z$-module, and not $p$-adically a $\ko$-module.
\end{Rmk}

\begin{Rmk}
In \cite{LMMT}, we show that the map $K(R) \to K(R[\tfrac{1}{p}])$ is a $K(1)$-local equivalence for all $K(1)$-acyclic ring spectra $R$, not only discrete rings (or $H\Z$-algebras) in which case it was first shown in \cite{BCM} using arithmetic techniques. Below, in \cref{Cor-invert-p}, we show that the same is true for the L-theory of connective ring spectra. It would be interesting to know whether this is true in full generality, and we plan to come back to this question in future work. 
\end{Rmk}

\section{Ring spectra}

In this section, we collect some results on telescopic localisations of the L-theory of ring spectra. The main ingredient is a formula we have learned from Yonatan Harpaz for certain relative L-theory terms, inspired by results of Weiss--Williams on the normal L-theory of the sphere spectrum \cite{WW3}, now incorporated in work with Nikolaus and Shah on real topological cyclic homology \cite{HNS}.

The following result implies \cref{Prop-TF} as well as \cref{Cor:finite-dimensional-rings}.

\begin{Thm}\label{Prop-general}
Let $R$ be a $T(n)$-acyclic $\E_1$-ring spectrum and let $\QF$ be a Poincar\'e structure on $\Perf(R)$ with underlying module with involution $M=R$. Then the map 
\[ \L^\q(R;M) \lto \L(R;\QF) \]
is a $T(n)$-local equivalence.
\end{Thm}
\begin{proof}
Since L-theory commutes with filtered colimits of Poincar\'e categories and $T(n)$-acyclic spectra are closed under colimits, it suffices to treat the case where the linear part of $\Qoppa$ is represented by a compact $R$-module $X$ -- see \cite[Section 3]{CDHI} for a general treatment of Poincar\'e structures on module categories. In this case, \cite{HNS} shows that the cofibre of the map in question is given by an equaliser of two maps
\[ \map_R(D_M(X),X) \lto (\S^{1-\sigma}\otimes \map(D_M(X),X))_{hC_2} \]
where $D_M(X) = \map_R(X,M)$. We claim that for all $Y$, the spectrum $\map_R(D_M(X),Y)$ is $T(n)$-acyclic: If $X=R$ this follows from the fact that $\map_R(M,Y)$ is an $R^\mathrm{op}$-module and hence $T(n)$-acyclic since $R$ (and hence $R^\mathrm{op}$) is $T(n)$-acyclic. In addition, the collection of $Z$ such that $\map_R(D_M(Z),Y)$ is $T(n)$-acyclic is closed under finite colimits and hence contains all perfect $R$-modules.
Therefore, both terms in the above equaliser are $T(n)$-acyclic, hence so is the equaliser itself.
\end{proof}

\begin{Rmk}
In the special case when $R$ is a $K(1)$-acyclic and connective $\E_\infty$-ring, and $\QF$ is the Tate Poincar\'e structure $\QF^\t$ (see again \cite[Example 3.2.12]{CDHI}), we can consider the maps
\[ \L^\q(\pi_0(R)) \longleftarrow \L^\q(R) \lto \L^\t(R) \]
in which the left map is an equivalence by \cite[Corollary 1.2.33 (i)]{CDHIII} and the right map is a $K(1)$-local equivalence by \cref{Prop-general}. Comparing to the same maps for $R[\tfrac{1}{p}]$, we deduce from \cref{Prop4} that the map $\L^\t(R) \to \L^\t(R[\tfrac{1}{p}])$ is a $K(1)$-local equivalence. However, it is in general not a $p$-adic equivalence for odd primes $p$: The fibre sequence of \cite{HNS} in this case reads as 
\[ \L^\q(R) \lto \L^\t(R) \lto \mathrm{Eq}\big( R \rightrightarrows R \otimes \mathrm{MTO(1)} \big) \]
so that the map $\L^\q(R) \to \L^\t(R)$ is a $p$-adic equivalence if $p$ is invertible in $R$, but not in general (e.g.\ for $R=\S$).

In particular, \cref{Prop4} is really something particular to the genuine Poincar\'e structures on discrete rings, whereas its consequence that the map $\L(R;\QF^{\geq m}_M) \to \L(R[\tfrac{1}{p}];\QF^{\geq m}_M)$ is a $K(1)$-local equivalence holds true more generally.
\end{Rmk}

\begin{Cor}\label{Cor:higher-vanishing-L}
Let $R$ be a connective and $T(n)$-acyclic $\E_1$-ring and $\QF$ a Poincar\'e structure with underlying module with involution $M$. If $n=1$, we have
\[ L_{K(1)}\L(R;\QF) \simeq \begin{cases} 0 & \text{ if } p=2 \\ \L^\s(\pi_0(R);M)^\cwedge_p & \text{ if } p \neq 2 \end{cases}\]
and if $n\geq 2$, we have that $L_{T(n)}\L(R;\QF)$ vanishes.
\end{Cor}
\begin{proof}
By \cref{Prop-general} it suffices to prove the claim for $\QF=\QF^\q_M$. In this case, the algebraic $\pi$-$\pi$-theorem, \cite[Corollary 1.2.33 (i)]{CDHIII}, shows that $\L^\q(R;M) \to \L^\q(\pi_0(R);M)$ is an equivalence. Here, by abuse of notation we also write $M$ for the module with involution over $\pi_0(R)$ induced by $M$. The result then follows from \cref{Theorem1}.
\end{proof}

\begin{Example}
We obtain that $\L^\s(\ku)$ is $T(2)$-acyclic, and hence by Hahn's result $T(n)$-acyclic for all $n\geq 2$ \cite{Hahn}. Likewise, iterated $\L$-theory $\L^{(n)}(\Z) = \L^\s(\dots(\L^\s(\Z)))$ remains $T(2)$-acyclic for all $n\geq 2$.
\end{Example}

\begin{Rmk}
In other words, L-theory does not redshift. This was of course well-known at very low heights, as the L-theory of $\F_p$-algebras is rationally trivial. The above shows that this behaviour does not change in higher heights and for ring spectra, which was of course also to be expected.
\end{Rmk}

\begin{Cor}\label{Cor-invert-p}
Let $R$ be a $K(1)$-acyclic and connective $\E_1$-ring spectrum with Poincar\'e structure $\QF$. Then the maps
\[ \L(R;\QF) \lto \L(R[\tfrac{1}{p}];\QF) \quad \text{ and }\quad \GW(R;\QF) \lto \GW(R[\tfrac{1}{p}];\QF)\]
are $K(1)$-local equivalences.
\end{Cor}
\begin{proof}
We have already seen the argument for the case of L-theory: By \cref{Prop-general} and the $\pi$-$\pi$-theorem \cite[Corollary 1.2.24]{CDHIII}, we may assume that $R$ is discrete. In this case the result was obtained in \cref{Prop4} for odd primes and at prime 2 all terms vanish $K(1)$-locally by \cref{Theorem1}. To deduce the case for GW, we again consider the diagram
\[\begin{tikzcd}
	K(R)_{hC_2} \ar[r] \ar[d] & \GW(R;\QF) \ar[r] \ar[d] & \L(R;\QF) \ar[d] \\
	K(R[\tfrac{1}{p}])_{hC_2} \ar[r] & \GW(R[\tfrac{1}{p}];\QF) \ar[r] & \L(R[\tfrac{1}{p}];\QF)
\end{tikzcd}\]
We have just argued that the right vertical map is a $K(1)$-equivalence, and \cite{LMMT} shows that the same is true for the left vertical map.
\end{proof}

\begin{Cor}\label{Cor-E-infty-connective}
Let $R$ be a connective and $T(n)$-acyclic ring spectrum with Poincar\'e structure $\QF$ and $n\geq 2$ or $n=1$ and $p=2$. Then of the following maps
\[ K(R)_{hC_2} \lto \GW(R;\QF) \lto K(R)^{hC_2} \, ,\]
the first one is a $T(m)$-equivalence for $m\geq n$ (with $p=2$ if $m=n=1$). If $(\Perf(R),\QF)$ is an associative algebra in $\mathrm{Cat}_\infty^\mathrm{p}$, then all of the above maps are $T(m)$-equivalences. If in addition the ring spectrum $R$ is $\E_\infty$, then all terms vanish $T(m)$-locally for $m \geq n+2$. 
\end{Cor}

We briefly comment on the assumptions in \cref{Cor-E-infty-connective}. Since the forgetful functor $\Cat_\infty^\mathrm{p} \to \Cat_\infty^\mathrm{ex}$ is canonically symmetric monoidal \cite[Theorem 5.2.7 (ii)]{CDHI}, $(\Perf(R),\QF)$ being an algebra implies that $\Perf(R)$ is a stably monoidal $\infty$-category, i.e.\ a monoidal and stable $\infty$-category whose tensor product commutes with finite colimits in each variable. If the unit is required to be $R$, then $R$ carries an induced $\E_2$-structure, and conversely any $\E_2$-structure on $R$ gives rise to a stably monoidal structure on $\Perf(R)$ with unit $R$. In addition, for $(\Perf(R),\QF)$ to be an algebra, one needs the Poincar\'e structure $\QF \colon \Perf(R)^\op \to \Sp$ is lax monoidal in a way that the induced lax monoidal structure on the duality $D\colon \Perf(R)^\op \to \Perf(R)$ is (strong) monoidal, see \cite[Corollary 5.3.18]{CDHI}.

\begin{proof}[Proof of \cref{Cor-E-infty-connective}]
Let $m\geq 2$. The first map is a $T(m)$-equivalence since its cofibre is $\L(R;\QF)$ which is $T(m)$-acyclic by \cref{Cor:higher-vanishing-L}. The cofibre of the composite is $K(R)^{tC_2}$ which, under the assumption that $(\Perf(R),\QF)$ is monoidal, is a module over $\L(R;\QF)$ and is hence also $T(m)$-acyclic. It remains to shown the  $T(m)$-local vanishing for $m\geq n+2$. By the first part, it suffices to show that the first term vanishes $T(m)$-locally. By the assumption that $R$ is $\E_\infty$, we find deduce from Hahn's result that $R$ is $T(m)$-acyclic for all $m \geq n$. The claim then follows from the combined redshift result of \cite{LMMT,CMNN} which given an equivalence 
\[ L_{T(m)}K(A) \simeq L_{T(m)} K(L_{T(m)\oplus T(m-1)}A).\]
\end{proof}

\begin{Rmk}
\cref{Cor-E-infty-connective} can have implications for nonconnective rings $R$ as well: For any $\E_2$-ring spectrum $R$, there is an $\E_2$-map $\tau_{\geq 0}R \to R$. If the monoidal functor $\Perf(\tau_{\geq 0}R) \to \Perf(R)$ of stable $\infty$-categories refines to a monoidal functor of Poincar\'e categories, we obtain that $\L(R)$ and $K(R)^{tC_2}$ are modules over $\L(\tau_{\geq 0} R)$, so we can use the vanishing results for $\L(\tau_{\geq 0} R)$ obtained by \cref{Cor-E-infty-connective} to obtain results for $\L(R)$ and $K(R)^{tC_2}$. However, the condition that $\Perf(\tau_{\geq 0} R) \to \Perf(R)$ refines to a monoidal functor of Poincar\'e categories is not always fulfilled. In fact, it is not even automatic that $\Perf(\tau_{\geq0}R) \to \Perf(R)$ refines to a Poincar\'e functor (disregarding multiplicative structures).
\end{Rmk}

\begin{Rmk}
Let $\KGW$ be the Karoubi localisation of $\GW$ as in \cite{CDHIV}. This is an invariant which sits in a natural fibre sequence
\[ \GW(\C,\QF) \lto \KGW(\C,\QF) \lto (\tau_{<0}\mathbb{K}(\C))^{hC_2}\]
where $\mathbb{K}$ denotes non-connective $K$-theory. We therefore see that the map $\GW(\C,\QF) \to \KGW(\C,\QF)$ is an equivalence on connective covers and hence a $T(n)$-local equivalence for all $n\geq 1$. Therefore, any result about telescopically localised $\GW$ applies equally to $\KGW$.
\end{Rmk}

\begin{Rmk}
Likewise, let $\KL$ be the Karoubi localisation of $\L$ as in \cite{CDHIV}. This is an invariant which sits in a natural fibre sequence
\[ \L(\C,\QF) \lto \KL(\C,\QF) \lto (\tau_{<0}\mathbb{K}(\C,\QF))^{tC_2},\]
and the map $\L \to \KL$ is a symmetric monoidal transformation between lax symmetric monoidal functors $\mathrm{Cat}_\infty^\mathrm{p} \to \mathrm{Sp}$.
In addition, we observe that 
\[ \L(\C,\QF) \lto \KL(\C,\QF) \]
is an equivalence after inverting 2, hence in particular a $p$-adic equivalence for odd primes $p$: Its cofibre is
\[ (\tau_{<0}\mathbb{K}(\C,\QF))^{tC_2} \simeq \colim_{n\to \infty} (\tau_{[-1,-n]}\mathbb{K})^{tC_2} \]
by \cite[Lemma I.2.6]{NS}. Since $\Z^{tC_2}$ is $\S[\tfrac{1}{2}]$-acyclic for $n\geq 1$, we deduce inductively that the terms appearing in the colimit on the right hand side are $\S[\tfrac{1}{2}]$-acyclic. Hence, the colimit is also $\S[\tfrac{1}{2}]$-acyclic as claimed.

We deduce that $\L(\C,\QF) \to \KL(\C,\QF)$ is a $T(n)$-local equivalence for $n\geq 1$ at odd primes, and using the monoidality of the transformation $\L \to \KL$, it also follows that all $K(1)$-local vanishing results at prime 2 in this note hold for $\KL$ in place of $\L$.
\end{Rmk}

\begin{Acknowledgements}
I want to thank Thomas Nikolaus for helpful suggestions on an early version of this document and Yonatan Harpaz for explaining to me the formula for relative L-theory.
\end{Acknowledgements}

\bibliographystyle{amsalpha}
\bibliography{remarks-on-GW}

\end{document}